\documentclass[reqno]{amsart}
\usepackage[dvips]{color}

\usepackage{amsmath,amssymb}
\usepackage{verbatim}
\theoremstyle{plain}
\newtheorem{lemma}{Lemma}[section]
\newtheorem{theorem}[lemma]{Theorem}

\newtheorem{proposition}[lemma]{Proposition}

\allowdisplaybreaks

\theoremstyle{remark}
\newtheorem{remark}{Remark}

\makeatletter
\newcommand*{\rom}[1]{\expandafter\@slowromancap\romannumeral #1@}
\makeatother

\numberwithin{equation}{section}
\begin{document}

\vskip 0.125in

\title[Ill-posedness and Blowup of the Inviscid Primitive Equations]
{Finite-time Blowup and Ill-posedness in Sobolev Spaces of the Inviscid Primitive Equations with Rotation}

\date{September 16, 2020}

\author[S. Ibrahim]{Slim Ibrahim}
\address[S. Ibrahim]
{Department of Mathematics and Statistics  \\
University of Victoria  \\
3800 Finnerty Road, Victoria  \\
 B.C., Canada V8P 5C2.} \email{ibrahims@uvic.ca}

\author[Q. Lin]{Quyuan Lin}
\address[Q. Lin]
{Department of Mathematics  \\
	Texas A\&M University  \\
	College Station  \\
	Texas, TX 77840, USA.} \email{abellyn@hotmail.com}

\author[E.S. Titi]{Edriss S. Titi}
\address[E.S. Titi]
{Department of Applied Mathematics and Theoretical Physics\\ University of Cambridge\\
Wilberforce Road, Cambridge CB3 0WA, UK.
Department of Mathematics  \\
	Texas A\&M University  \\
	College Station  \\
	Texas, TX 77840, USA.
 Department of Computer Science and Applied Mathematics \\
Weizmann Institute of Science  \\
Rehovot 76100, Israel.} \email{Edriss.Titi@damtp.cam.ac.uk}
\email{titi@math.tamu.edu}

\begin{abstract}
Large scale dynamics of the oceans and the atmosphere are governed by the primitive equations (PEs). It is well-known that the three-dimensional viscous PEs is globally well-posed in Sobolev spaces. On the other hand, the inviscid PEs without rotation is known to be ill-posed in Sobolev spaces, and its smooth solutions can form singularity in finite time. In this paper, we extend the above results in the presence of rotation. First, we construct finite-time blowup solutions to the inviscid PEs with rotation, and establish that the inviscid PEs with rotation is ill-posed in Sobolev spaces in the sense that its perturbation around a certain steady state background flow is both linearly and nonlinearly ill-posed in Sobolev spaces. Its linear instability is of the Kelvin-Helmholtz type similar to the one appears in the context of vortex sheets problem. This implies that the inviscid PEs is also linearly ill-posed in Gevrey class of order $s > 1$, and suggests that a suitable space for the well-posedness is Gevrey class of order $s = 1$, which is exactly the space of analytic functions.
\end{abstract}

\maketitle



\section{introduction}
We consider the following $3D$ primitive equations (PEs):
\begin{eqnarray}
	&&\hskip-.8in u_t + u u_x + vu_x + w u_z -\nu_h \Delta u - \nu_z u_{zz} -\Omega \, v+
	p_x = 0 , \label{EQ1-1}  \\
	&&\hskip-.8in v_t + u v_x + vv_y + w
	v_z -\nu_h \Delta v - \nu_z v_{zz} +\Omega\, u + p_y = 0 , \label{EQ1-2}  \\
	&&\hskip-.8in
	p_z + T =0 ,   \label{EQ1-3}  \\
	&&\hskip-.8in
	T_t + u T_x + v T_y + w T_z -\kappa_h \Delta T - \kappa_z \partial_{zz} T = 0, \label{EQ1-4} \\
	&&\hskip-.8in
	u_x + v_y+ w_z =0   \label{EQ1-5}
\end{eqnarray}
in the horizontal channel $\big\{(x,y,z): 0\leq z\leq 1, (x,y)\in \mathbb{R}^2\big\}$, which are supplemented with the initial value $(u_0, v_0, T_0)$, and satisfy the relevant geophysical boundary conditions (cf. \cite{CT07,LTW92a,LTW92b,LTW95}). Here the horizontal velocity $(u,v)$, the vertical velocity $w$, the temperature $T$, and the pressure $p$ are the unknown functions of the variable $(t,x,y,z)$. The parameters $\nu_h \geq 0$ and $\nu_z \geq 0$ denote to the horizontal and vertical viscosities, $\kappa_h \geq 0$ and $\kappa_z \geq 0$ denote to the horizontal and vertical diffusivities, respectively. The parameter $\Omega\in\mathbb{R}$ denotes to the Coriolis parameter, which indicates the rate and direction of rotation. We denote by $\Delta = \partial_{xx} + \partial_{yy}$ the $2D$ horizontal Laplacian. When $\nu_h, \nu_z, \kappa_h, \kappa_z >0$, system (\ref{EQ1-1})--(\ref{EQ1-5}) is derived as a formal asymptotic limit of the small aspect ratio (the ratio of the depth or the height to the horizontal length scale) from the Rayleigh-B\'enard (Boussinesq) system. The derivation was rigorously justified first by Az\'erad and Guill\'en \cite{AG01} in a weak sense, then by Li and Titi \cite{LT18} in a strong sense with error estimates in terms of the small aspect ratio. The global existence of strong solutions for the $3D$ PEs with full viscosity and full diffusion was first established by Cao and Titi in \cite{CT07}, and later by Kobelkov in \cite{K06}, see also the subsequent articles of Kukavica and Ziane \cite{KZ07,KZ072} for different boundary conditions,  as well as Hieber and Kashiwabara \cite{Hieber-Kashiwabara} for some progress towards relaxing the smoothness on the initial data by using the semigroup method. With only horizontal viscosity, i.e., $\nu_h >0$ and $\nu_z =0$, the global well-posedness of $3D$ PEs was established by Cao, Li and Titi in \cite{CLT16,CLT17,CLT17b}. On the other hand, with only vertical viscosity, i.e., $\nu_h =0$ and $\nu_z>0$, Cao, Lin and Titi established recently \cite{CLT19} the local well-posedness of the PEs in Sobolev spaces by considering an additional weak dissipation, which is the linear (Rayleigh-like friction) damping. Following \cite{CLT19}, it can be shown that with this linear damping, the inviscid PEs ($\nu_h =0$ and $\nu_z=0$) is also locally well-posed in Sobolev spaces, with higher Sobolev regularity requirement on the initial data. This linear damping helps the system overcome the ill-posedness in Sobolev spaces (see discussion below) established by Renardy in \cite{RE09} for the case without rotation, and current paper for the case with rotation. See also \cite{CT10} for a similar idea on the effect of this linear damping.

When $\nu_h = \nu_z =0$, the inviscid PEs without coupling with the temperature is also called the hydrostatic Euler equations. In the absence of rotation ($\Omega = 0$), the linear ill-posedness of the inviscid PEs, near certain background shear flows, has been established by Renardy in \cite{RE09}. Later on, the nonlinear ill-posedness of the inviscid PEs without rotation was established by Han-Kwan and Nguyen in \cite{HN16}, where they built an abstract framework to show the inviscid PEs are ill-posed in any Sobolev space. Moreover, it was proven that smooth solutions to the inviscid PEs, in the absence of rotation, can develop singularities in finite time. (cf. Cao, Ibrahim, Nakanishi and Titi \cite{CINT15}, and Wong \cite{W12}.) However, the results mentioned above do not include the case when the rotation rate $\Omega \neq 0$, which is the subject matter of this paper.

The linear ill-posedness results show that the linearized $2D$ inviscid PEs (which implies the same results for $3D$ case, see details below), around a special steady state background flow, has unstable solutions of the form $u(t,x,z) = e^{2\pi ikx} e^{\sigma_k t} u_k(z)$, where $\Re \sigma_k = \lambda k$ for some $\lambda \in \mathbb{R}$. Such Kelvin-Helmholtz type instability forbids the construction of solutions in Sobolev spaces. Kelvin-Helmholtz type instability also appears in the context of vortex sheets, see, e.g., \cite{CO89} (see also the survey paper \cite{BT07} and reference therein). To obtain positive results, one must start from initial data $u_0$ that are strongly
localized in Fourier, typically for which $|\hat{u}_0(k,z)|\lesssim e^{-\delta |k|^{1/s}}$ with $\delta >0$ and $s\geq 1$. Such localization condition corresponds to Gevrey class of order $s$ in the $x$ variable. Kelvin-Helmholtz type instability also forbids the construction of solutions in Gevery class of order $s>1$. This suggests that the suitable space for the well-posedness of the inviscid PEs (with or without rotation) is Gevrey class of order $s=1$, which is the space of analytic functions. This is consistent with positive results in \cite{ILT20,KTVZ11}. Notably, for the Prandtl equations, which has some similarity in its structure with the PEs, is shown in \cite{GD10} that its linearization around a special background flow has unstable solutions of similar form, but with $\Re \sigma_k \sim \lambda \sqrt{k}$ for $k\gg 1$ arbitrarily large and some positive $\lambda \in \mathbb{R}_+$. This implies that the optimal order $s$ for Prandtl equation is $s=2$, which is consistent with the positive results in \cite{DG19,LMY20}. This shows that the linear instability of the inviscid PEs is ``worse" than that of the Prandtl equations.

Although it is ill-posed in Sobolev spaces, the well-posedness of the inviscid PEs can be obtained by assuming either real analyticity
or some special structures (local Rayleigh condition) on the initial data \cite{BR99,BR03,GR99,KMVW14,KTVZ11,MW12}. In particular, the authors in \cite{KTVZ11} establish the local well-posedness in time of the $3D$ inviscid PEs in the space of analytic functions, but the time of existence they obtained shrinks to zero as the rate of rotation $|\Omega|$ increases toward infinity. This is contrary to the cases of the $3D$ fast rotating Euler, Navier–Stokes and Boussinesq equations, where the limit of fast rotation leads to strong ``dispersion" or averaging mechanism that weakens the nonlinear effects allowing to establishing the global regularity result in the case of the Navier-Stokes equations, and prolonging the life-span of the solution in the case of Euler equations, by Babin, Mahalov and Nicolaenko \cite{BMN97, BMN99a, BMN99b, BMN00} (see also \cite{CDGG06,D05,EM96,IY,KLT14} and references therein). In addition, we refer to \cite{BIT11,GST15,KTZ18,LT04} for simple examples demonstrating the above mechanism. In \cite{ILT20}, we improve the results in \cite{KTVZ11} by establishing the local in time well-posedness in the space of analytic functions for a time interval that is independent of the rate of rotation. Furthermore, we also establish a lower bound on the life-span of the solution that grows to infinity with the rotation rate for ``well-prepared" initial data. In a sense the results reported in this paper furnish a solid justification and motivation for our study in \cite{ILT20}. Specifically, the purpose of this paper is to establish the finite-time blowup and the ill-posedness in Sobolev spaces of the $3D$ inviscid PEs with rotation, and to investigate the effect of rotation on the blowup time. Due to the linear ill-posedness in Sobolev spaces and Gevrey class of order $s>1$, it is natural to consider the question of well-posedness of the inviscid PEs (with or without rotation) in Gevrey class of order $s=1$, which is the space of analytic functions (cf. \cite{ILT20} and \cite{KTVZ11}). By virtue of the finite-time blowup results, one can conclude that there is no hope to show the global well-posedness of the $3D$ inviscid PEs, even with fast rotation. The optimal result one can expect is that fast rotation prolongs the life-span of the $3D$ inviscid PEs (cf. \cite{ILT20}).

For this endeavor, we first simplify system (\ref{EQ1-1})--(\ref{EQ1-5}). If initially $T_0 = 0$, then it is easy to see that any smooth solution $(u,v,w,T)$ to system (\ref{EQ1-1})--(\ref{EQ1-5}), with initial data $(u_0, v_0, T_0)$ and suitable boundary conditions, must satisfy $T(t,x,y,z) \equiv 0$. Moreover, if initially $u_0$ and $v_0$ are independent of the $y$ variable, then any smooth solution $(u,v,w,T)$ remains independent of the $y$ variable. Notably, this is not true, e.g., in the case of weak wild solutions of the Euler equations (cf. \cite{BLNNT13}). Therefore, under these assumptions on the initial data, namely,
\begin{equation} \label{2d-condition}
    T_0(x,y,z)=T_0(x,z) = 0, \;\; u_0(x,y,z)=u_0(x,z), \;\; v_0(x,y,z)=v_0(x,z),
\end{equation}
we obtain the reduced (adiabatic) inviscid PEs system:
\begin{eqnarray}
	&&\hskip-.8in u_t + u\, u_x + w u_z -\Omega \, v+
	p_x = 0 , \label{EQ2-1}  \\
	&&\hskip-.8in v_t + u\, v_x + w
	v_z +\Omega\, u = 0 , \label{EQ2-2}  \\
	&&\hskip-.8in
	p_z =0 ,   \label{EQ2-3}  \\
	&&\hskip-.8in
	u_x + w_z =0   \label{EQ2-4}
\end{eqnarray}
in the horizontal channel $\big\{(x,z): 0\leq z\leq 1, x\in \mathbb{R}\big\}$. We supplement the above system with the initial condition $(u_0, v_0)$, subject to no-normal flow boundary condition on the top and bottom, and periodic in the $x$ variable:
\begin{equation}\label{BC-1}
    \begin{split}
        &w(t,x,0) = w(t,x,1) = 0, \\
        &u, v, w, p \;\text{are periodic in } x \;\text{with period } 1.
    \end{split}
\end{equation}
Observe that system (\ref{EQ2-1})--(\ref{EQ2-4}) is the reduced $3D$ hydrostatic Euler equations with rotation. The finite-time blowup of solutions or the ill-posedness of system (\ref{EQ2-1})--(\ref{EQ2-4}) imply the blowup of solutions or the ill-posedness of solutions of the original $3D$ system (\ref{EQ1-1})--(\ref{EQ1-5}) in the inviscid case ($\nu_h=\nu_z = 0$), with initial data satisfying (\ref{2d-condition}). Therefore, we focus in this paper on system (\ref{EQ2-1})--(\ref{EQ2-4}).

The paper is organized as follows. In section 2, we adopt ideas from \cite{CINT15,W12} and consider the initial data depending on the rate of rotation to establish the finite-time blowup of solutions to system (\ref{EQ2-1})--(\ref{EQ2-4}). In section 3, we show the ill-posedness of the perturbation about a steady state background flow which depends on $\Omega$ and has infinite energy. The perturbations about this background steady state are assumed to satisfy the period boundary conditions (\ref{BC-1}). In particular, following \cite{HN16,RE09}, we establish that the perturbed system about this steady state background flow is both linearly and nonlinearly ill-posed in Sobolev spaces, and is linearly ill-posed in Gevrey class of order $s>1$. In section 4, we make some concluding remarks, and study the effect of rotation on the blowup time based on the results established in section 2. Finally, We propose some interesting problems for future study.

\section{blowup of solutions}
In this section, we adopt the ideas in \cite{CINT15,W12}. We assume that $u$ and $v$ are odd in the $x$ variable, and that $w$ and $p$ are even in the $x$ variable. Observe that such symmetric conditions are invariant under smooth dynamics of system (\ref{EQ2-1})--(\ref{EQ2-4}). We first introduce the following proposition from \cite{CINT15} and provide further analysis strengthening its conclusion. Observe that in \cite{CISY89} (see also \cite[section 4]{O09}, and references therein), a similar problem, arising in a different fluid dynamic context, has been investigated.
\begin{proposition}\label{theorem-phi}(see \cite{CINT15})
Consider the following nonlinear nonlocal degenerate elliptic boundary value problem:
\begin{eqnarray}
\phi' - (\phi')^2 + \phi \phi^{''} + 2 \int_0^1 (\phi'(z))^2 dz = 0 ,\;\; \phi(0)=\phi(1)=0. \label{phi}
\end{eqnarray}
Then for each $\alpha\in(0,1)$, the boundary value problem (\ref{phi}) has a nontrivial solution $\phi_\alpha \in C^{2,\alpha}([0,1])$.
\end{proposition}
Recall that, for an integer $k$, and $0<\alpha<1$ the space $C^{k,\alpha}$ is endowed with the norm
$$
\|f\|_{C^{k,\alpha}}=\|f\|_{C^k}+\sup_{x\neq y}\frac{|f(x)-f(y)|}{|x-y|^\alpha}.
$$
\begin{proof}
  For each $m>0$, the existence of nontrivial solution to the boundary value problem (\ref{phi}) satisfying the additional constraint
  \begin{equation}\label{constraint-m}
      2\int_0^1 (\phi'(z))^2 dz = m^2
  \end{equation} has been established in \cite{CINT15}. Let $\alpha\in(0,1)$ and define
\begin{equation}
     m : = \sqrt{\Big(\frac{1+\alpha}{2(1-\alpha)}\Big)^2 - \frac{1}{4}}>0.
\end{equation}
That is

\begin{equation}
    \alpha = \frac{\sqrt{m^2+1/4} - \frac{1}{2}}{\sqrt{m^2+1/4} + \frac{1}{2}}.
\end{equation}
Denoting by $\psi := \phi'$, it was shown in \cite{CINT15} that the nontrivial solution $\phi$ of problem (\ref{phi}) satisfying (\ref{constraint-m}) can be written as
  \begin{equation}\label{phi-solution}
      \phi = C(m)(\psi_+ - \psi)^{\frac{\psi_+}{\psi_+-\psi_-}} (\psi- \psi_-)^{\frac{-\psi_-}{\psi_+-\psi_-}},
  \end{equation}
where
\begin{equation}
    \psi_\pm(m) := \pm \sqrt{m^2 + 1/4} + 1/2,
\end{equation}
and
\begin{equation}\label{cm}
    C(m) = \frac{1}{B(\frac{\psi_+}{\psi_+-\psi_-} ,\frac{-\psi_-}{\psi_+-\psi_-} )}.
\end{equation}
Here $B(a,b) = \int_0^1 t^{a-1} (1-t)^{b-1} dt$ is the Beta function. Moreover, it was also shown in \cite{CINT15} that $\phi$, $\psi$ and $z$ satisfy
\begin{equation}\label{boundary-psipm}
    (\phi,\psi)(z=0) = (0,\psi_+),\;\;\;\; (\phi,\psi)(z=1) = (0,\psi_-), \;\;\;\; \psi_-\leq \psi \leq \psi_+,
\end{equation}
and
\begin{equation}\label{dpsi-dz}
    \frac{d\psi}{dz} = \frac{-1}{C(m)} (\psi_+ - \psi)^{\frac{-\psi_-}{\psi_+-\psi_-}} (\psi- \psi_-)^{\frac{\psi_+}{\psi_+-\psi_-}}.
\end{equation}
Therefore, $\psi$ is a continuous and decreasing function of $z$, and smooth in $(0,1)$. From (\ref{boundary-psipm}) and (\ref{dpsi-dz}), one has
\begin{equation}\label{z-psi}
    z(\psi) = -C(m) \int_{\psi_+}^\psi (\psi_+ - \tilde{\psi})^{\frac{\psi_-}{\psi_+-\psi_-}} (\tilde{\psi}- \psi_-)^{\frac{-\psi_+}{\psi_+-\psi_-}} d\tilde{\psi},
\end{equation}
and
\begin{equation}\label{z-psi-2}
    z(\psi) -1 = -C(m) \int_{\psi_-}^\psi (\psi_+ - \tilde{\psi})^{\frac{\psi_-}{\psi_+-\psi_-}} (\tilde{\psi}- \psi_-)^{\frac{-\psi_+}{\psi_+-\psi_-}} d\tilde{\psi}.
\end{equation}
Next, we establish that $\phi(z)\in  C^{2,\alpha}([0,1])$. From (\ref{phi-solution}) and (\ref{dpsi-dz}), we know when $\psi$ is away from $\psi_+$ and $\psi_-$, i.e., $z$ is away from $0$ and $1$, $\phi(z)$ is smooth. Therefore, we only need to consider when $\psi$ is close to $\psi_+$ and $\psi_-$. From (\ref{z-psi}), one has
\begin{equation}
    z(\psi) = C(m) B(\frac{\psi_+-\psi}{\psi_+-\psi_-};\frac{\psi_+}{\psi_+-\psi_-} ,\frac{-\psi_-}{\psi_+-\psi_-} ),
\end{equation}
where $B(x; a, b) = \int_0^x t^{a-1} (1-t)^{b-1} dt$ is the incomplete Beta function ($0\leq x\leq 1$). Moreover, from (\ref{cm}), we know that
\begin{equation}
    z(\psi) = \frac{ B(\frac{\psi_+-\psi}{\psi_+-\psi_-};\frac{\psi_+}{\psi_+-\psi_-} ,\frac{-\psi_-}{\psi_+-\psi_-} )}{B(\frac{\psi_+}{\psi_+-\psi_-} ,\frac{-\psi_-}{\psi_+-\psi_-})} = I( \frac{\psi_+-\psi}{\psi_+-\psi_-}; \frac{\psi_+}{\psi_+-\psi_-} ,\frac{-\psi_-}{\psi_+-\psi_-}),
\end{equation}
where $I(x;a,b) = \frac{B(x;a,b)}{B(a,b)}$ is the regularized Beta function. When $x\in(0,1)$, by series expansion (cf. \cite[p.~944]{AS73}), one has
\begin{equation}
    I(x;a,b) = \frac{x^a (1-x)^b}{aB(a,b)}\Big\{1+ \sum\limits_{n=0}^\infty \frac{B(a+1,n+1)}{B(a+b,n+1)}x^{n+1}  \Big\}.
\end{equation}
Therefore, for $\psi_- < \psi < \psi_+$, we can write
\begin{equation}\label{z-series-1}
    z(\psi) = \frac{C(m)}{\psi_+} (\psi_+-\psi)^{\frac{\psi_+}{\psi_+-\psi_-}}(\psi-\psi_-)^{\frac{-\psi_-}{\psi_+-\psi_-}} \Big\{ 1+ \sum\limits_{n=0}^\infty \frac{B(\frac{\psi_+}{\psi_+-\psi_-}+1,n+1)}{B(1,n+1)}\Big(\frac{\psi_+-\psi}{\psi_+-\psi_-}\Big)^{n+1}   \Big\}.
\end{equation}
Letting
\begin{equation}
    h_1(\psi) : = \sum\limits_{n=0}^\infty \frac{B(\frac{\psi_+}{\psi_+-\psi_-}+1,n+1)}{B(1,n+1)}\Big(\frac{\psi_+-\psi}{\psi_+-\psi_-}\Big)^{n+1},
\end{equation}
then $h_1(\psi)\geq 0$ and $h_1(\psi)$ is smooth on $\psi\in (\psi_-, \psi_+]$. Combine (\ref{dpsi-dz}) and (\ref{z-series-1}), we find that for $z\in(0,1)$,
\begin{equation}\label{dpsi-dz-2}
    \frac{d\psi}{dz} = -C(m)^{\frac{\psi_- - \psi_+}{\psi_+}}\Big(\frac{\psi_+}{1+h_1(\psi(z))}\Big)^{\frac{-\psi_-}{\psi_+}}\Big(\psi(z) - \psi_-\Big)^{\frac{\psi_+ + \psi_-}{\psi_+}} z^{\frac{-\psi_-}{\psi_+}}.
\end{equation}
From this expression and since $h_1(\psi(z)))$ is smooth on $z\in[0,1)$, we conclude that $\frac{d\psi}{dz}$ is continuous on $z\in[0,1)$, and smooth on $z\in(0,1)$. Observe that $\alpha = \frac{-\psi_-}{\psi_+}$, thus we have
\begin{equation}
\begin{split}
    \lim\limits_{z\rightarrow 0^+}\frac{|\frac{d\psi}{dz}(z) - \frac{d\psi}{dz}(0)|}{|z-0|^{\alpha}} &= \lim\limits_{z\rightarrow 0^+} C(m)^{\frac{\psi_- - \psi_+}{\psi_+}}\Big(\frac{\psi_+}{1+h_1(\psi(z))}\Big)^{\frac{-\psi_-}{\psi_+}}\Big(\psi(z) - \psi_-\Big)^{\frac{\psi_+ + \psi_-}{\psi_+}}\\
    & = \lim\limits_{\psi\rightarrow \psi_+} C(m)^{\frac{\psi_- - \psi_+}{\psi_+}}\Big(\frac{\psi_+}{1+h_1(\psi)}\Big)^{\frac{-\psi_-}{\psi_+}}(\psi - \psi_-)^{\frac{\psi_+ + \psi_-}{\psi_+}}\\
    & = C(m)^{\frac{\psi_- - \psi_+}{\psi_+}} \psi_+^{\frac{-\psi_-}{\psi_+}}(\psi_+ - \psi_-)^{\frac{\psi_+ + \psi_-}{\psi_+}} < \infty.
\end{split}
\end{equation}
Therefore, $\psi(z) \in C^{1,\alpha}([0,1))$, and thus $\phi(z)\in C^{2,\alpha}([0,1))$. Similarly, from (\ref{z-psi-2}), for $\psi_- < \psi < \psi_+$, we can write
\begin{equation}\label{z-series-2}
\begin{split}
        1-z(\psi) &= I(\frac{\psi-\psi_-}{\psi_+-\psi_-}; \frac{-\psi_-}{\psi_+-\psi_-}, \frac{\psi_+}{\psi_+-\psi_-}) \\
        & = \frac{-C(m)}{\psi_-} (\psi-\psi_-)^{\frac{-\psi_-}{\psi_+-\psi_-}} (\psi_+-\psi)^{\frac{\psi_+}{\psi_+-\psi_-}} \Big\{ 1+ \sum\limits_{n=0}^\infty \frac{B(\frac{-\psi_-}{\psi_+-\psi_-}+1,n+1)}{B(1,n+1)}(\frac{\psi-\psi_-}{\psi_+-\psi_-})^{n+1}   \Big\}.
\end{split}
\end{equation}
Letting
\begin{equation}
    h_2(\psi) : = \sum\limits_{n=0}^\infty \frac{B(\frac{-\psi_-}{\psi_+-\psi_-}+1,n+1)}{B(1,n+1)}(\frac{\psi-\psi_-}{\psi_+-\psi_-})^{n+1},
\end{equation}
then $h_2(\psi)\geq 0$ and $h_2(\psi)$ is smooth on $\psi\in [\psi_-, \psi_+)$. Combine (\ref{dpsi-dz}) and (\ref{z-series-2}), we find that
\begin{equation}\label{dpsi-dz-3}
    \frac{d\psi}{dz} = -C(m)^{\frac{\psi_+ - \psi_-}{\psi_-}}\Big(\frac{-\psi_-}{1+h_2(\psi(z))}\Big)^{\frac{-\psi_+}{\psi_-}}\Big(\psi_+ - \psi(z)\Big)^{\frac{\psi_+ + \psi_-}{\psi_-}} (1-z)^{\frac{-\psi_+}{\psi_-}}.
\end{equation}
From this expression and since $h_2(\psi(z)))$ is smooth on $z\in(0,1]$, observe that $\frac{-\psi_+}{\psi_-} >1$ and since $\psi(z) \in C^{1,\alpha}([0,1))$, we know that indeed $\psi(z) \in C^{1,\alpha}([0,1])$. Therefore, $\phi(z)\in C^{2,\alpha}([0,1])$.
\end{proof}
Under the assumption about symmetry on the initial data in system (\ref{EQ2-1})--(\ref{EQ2-4}), we establish the following blowup result.
\begin{theorem}\label{theorem-CINT}
Let $\phi(z)$ be a nontrivial solution of the boundary value problem (\ref{phi}), and let $f(x)$ be a smooth odd periodic function with period 1, satisfying $f'(0)=1$. Suppose that $(u, v, w, p)$ is a smooth solution to system (\ref{EQ2-1})--(\ref{EQ2-4}), subject to the boundary
condition (\ref{BC-1}), with initial condition
\begin{equation}\label{IC-CINT}
   u_0(x, z) = -f(x) \phi'(z), \;\; v_0(x, z) = -\Omega f(x).
\end{equation}
Then the solution blows up at
sometime $\mathcal{T}\in(0,1]$.
\end{theorem}
Before proving this theorem, let us simplify system (\ref{EQ2-1})--(\ref{EQ2-4}) further. From (\ref{EQ2-4}) and (\ref{BC-1}), we know
\begin{equation} \label{ux-barotropic}
    \int_0^1 u_x(t,x,z) dz = 0.
\end{equation}
Differentiating (\ref{EQ2-1}) and (\ref{EQ2-2}) with respect to $x$, we have
\begin{eqnarray}
&&\hskip-.8in u_{xt} + u\, u_{xx} + u_x^2 + w_x u_z + wu_{xz} -\Omega \, v_x+
p_{xx} = 0 , \label{EQx-1}  \\
&&\hskip-.8in v_{xt} + u_x\, v_x  + uv_{xx} + w_x
v_z + wv_{xz} +\Omega\, u_x = 0 . \label{EQx-2}
\end{eqnarray}
Thanks to (\ref{ux-barotropic}), integrating (\ref{EQx-1}) with respect to $z$ over the interval $[0,1]$, an integration by parts together with (\ref{EQ2-3}),  (\ref{EQ2-4}) and (\ref{BC-1}), implies:
\begin{eqnarray}
&&\hskip-.8in  p_{xx} = \int_0^1 \Big[ -2(uu_x)_x + \Omega v_x \Big] dz. \label{EQpxx}
\end{eqnarray}
Let
$$W(t,z) =w(t,0,z), \;\;\; V(t,z) = -v_x(t,0,z).$$
Plugging (\ref{EQpxx}) back to (\ref{EQx-1}), and by virtue of the oddness of $u$ and $v$ (thanks to (\ref{EQ2-4})), system (\ref{EQx-1})--(\ref{EQx-2}) restricts on the line $x=0$ becomes
\begin{eqnarray}
&&\hskip-.8in
W_{tz} - (W_z)^2 + WW_{zz} + 2\int_0^1 W_z^2(t,z) dz - \Omega V + \Omega \int_0^1 V(t,z)dz= 0, \label{reduced-case1-1} \\
&&\hskip-.8in
V_{t} - W_z V + WV_{z} + \Omega W_z = 0. \label{reduced-case1-2}
\end{eqnarray}
By virtue of (\ref{BC-1}) and (\ref{IC-CINT}), the corresponding initial and boundary conditions are
\begin{equation}\label{reduced-case1-IC}
    W(0,z) = \phi(z), \;\; V(0,z) = \Omega,
\end{equation}
\begin{equation}
    W(t,0)=W(t,1) = 0. \label{reduced-case1-BC}
\end{equation}
We have the following proposition concerning the uniqueness of solutions to system (\ref{reduced-case1-1})--(\ref{reduced-case1-BC}).
\begin{proposition}\label{theorem-uniqueness}
  Let $(W_1, V_1)$ and $(W_2, V_2)$ be two solutions of the same initial boundary value problem  (\ref{reduced-case1-1})--(\ref{reduced-case1-BC}) satisfying the regularity $W_1, W_2 \in L^2(0,\mathcal{T}; H^2)$ and $V_1, V_2 \in L^2(0,\mathcal{T}; H^1)$, with the same initial data. Then $W_1 = W_2$ and $V_1=V_2$ for any $t\in [0,\mathcal{T})$.
\end{proposition}
\begin{proof}
Denote by $W = W_1 - W_2$, $\bar{W}= \frac{1}{2}(W_1+W_2)$, $V=V_1-V_2$, and $\bar{V}= \frac{1}{2}(V_1+V_2)$. Then (\ref{reduced-case1-1})--(\ref{reduced-case1-BC}) implies that
\begin{eqnarray}
&&\hskip-.8in
W_{tz} - 2\bar{W}_z W_z + W \bar{W}_{zz} + \bar{W}W_{zz} + 4\int_0^1 \bar{W}_z W_z dz - \Omega V + \Omega \int_0^1 V dz =0, \label{uniqueness-W}\\
&&\hskip-.8in
V_t - W_z \bar{V} - \bar{W}_z V + W\bar{V}_z + \bar{W}V_z + \Omega W_z = 0, \label{Uniqueness-V}
\end{eqnarray}
with boundary condition
\begin{equation}
    W(t,0)=W(t,1)=0. \label{Uniqueness-BC}
\end{equation}
Multiplying (\ref{uniqueness-W}) by $W_z$, and (\ref{Uniqueness-V}) by $V$, integrating with respect to $z$ over the interval $[0,1]$, then an integration by parts together with the boundary condition (\ref{Uniqueness-BC}) gives
\begin{eqnarray}
&&\hskip-.8in
\frac{1}{2}\frac{d}{dt}\Big(\|W_z\|^2_{L^2(0,1)}+ \|V\|^2_{L^2(0,1)}\Big)=\int_0^1 \Big(\frac{5}{2} \bar{W}_z W_z^2  - \bar{W}_{zz} W W_{z} +  W_z \bar{V} V  + \frac{3}{2} \bar{W}_z V^2 - W \bar{V}_z V\Big) dz.
\end{eqnarray}
Using H\"older inequality, Young's inequality, Sobolev inequality, and thanks to (\ref{Uniqueness-BC}), we can apply Poincar\'e inequality to obtain
\begin{eqnarray}
&&\hskip-.8in
\frac{d}{dt}\Big(\|W_z\|^2_{L^2(0,1)}+ \|V\|^2_{L^2(0,1)}\Big) \leq C\Big(\|\bar{W}\|_{H^2(0,1)} + \|\bar{V}\|_{H^1(0,1)}\Big)\Big(\|W_z\|^2_{L^2(0,1)}+ \|V\|^2_{L^2(0,1)}\Big).
\end{eqnarray}
Thanks to Gr\"onwall inequality, and since $W_z(0,z) = V(0,z) = 0$, for any $t\in[0,\mathcal{T})$, we have
\begin{eqnarray}
&&\hskip-.8in
\|W_z(t)\|^2_{L^2(0,1)}+ \|V(t)\|^2_{L^2(0,1)} \leq \Big(\|W_z(0)\|^2_{L^2(0,1)}+ \|V(0)\|^2_{L^2(0,1)}\Big) \nonumber\\
&&\hskip.8in
\times\exp\Big(C\int_0^t \|\bar{W}(s)\|_{H^2(0,1)} + \|\bar{V}(s)\|_{H^1(0,1)} ds\Big) = 0.
\end{eqnarray}
By Poincar\'e inequality, we conclude that $\|W(t)\|_{L^2(0,1)} = 0$. Thus, $W_1 = W_2$ and $V_1 = V_2$ for any $t\in[0,\mathcal{T})$.
\end{proof}
Now let us return to the proof of Theorem \ref{theorem-CINT}.
\begin{proof}(Proof of Theorem \ref{theorem-CINT})
From Proposition \ref{theorem-phi}, we conclude that $\phi(z)\in H^2(0,1)$. Thanks to Proposition \ref{theorem-uniqueness}, one can observe that
\begin{equation}\label{selfsimilar}
    W(t,z) = \frac{\phi(z)}{1-t}, \;\;\;\; V(t,z) \equiv \Omega
    \end{equation}
is the unique solution of (\ref{reduced-case1-1})--(\ref{reduced-case1-2}) subject to initial and boundary conditions (\ref{reduced-case1-IC})--(\ref{reduced-case1-BC}). Then we see $W(t,z)$ blows up at $t=1$, and therefore, the solution $(u,v,w,p)$ must blow up at sometime $\mathcal{T}\in (0,1]$.
\end{proof}
The above approach follows from \cite{CINT15}. In addition, we provide another approach that adopts ideas from \cite{W12}. This approach requires some additional conditions on the initial data, but avoids technical issue on the function $\phi$ as in Proposition \ref{theorem-phi}. Moreover, this approach allows the initial data to be analytic, which guarantees the existence of solutions to the inviscid PEs in the space of analytic functions thanks to the results in \cite{ILT20,KTVZ11}.
\begin{theorem}\label{theorem-Wong}
  Suppose that $(u, v, w, p)$ is a smooth solution to system (\ref{EQ2-1})--(\ref{EQ2-4}), subject to the boundary condition (\ref{BC-1}), with initial condition $(u_0,v_0)$ satisfying the following conditions:
\begin{equation}\label{IC-Wong-1}
   u_0(x,z) \; \text{and } v_0(x,z) \; \text{are smooth odd periodic functions in } x \; \text{with period } 1,
\end{equation}
\begin{equation}\label{IC-Wong-2}
   u_0(x,z)  \; \text{satisfies the compatibility condition } \int_0^1 u_0(x,z)dz = 0,
\end{equation}
\begin{equation}\label{IC-Wong-3}
   \partial_x v_0(0,z) = -\Omega \; \text{ for all } z\in[0,1],
\end{equation}
\begin{equation}\label{IC-Wong-4}
   \partial_{xz} u_0(0,0) = 0, \; \partial_{xzz} u_0(0,z)<0 \; \text{ for all } z\in[0,1].
\end{equation}
Then the solution blows up at
sometime $\mathcal{T}\in(0,\frac{-3}{\partial_x u_0(0,1)}]$.
\end{theorem}
Before proving this theorem, we first state the following lemma, which is Lemma 2.4 in \cite{W12}. Since our settings are slightly different from \cite{W12}, we provide a detailed proof in here.
\begin{lemma}\label{theorem-lemma-Wong-invariant}(see \cite{W12} Lemma 2.4)
  The smooth solution $(u,v,w,p)$ stated in Theorem \ref{theorem-Wong} satisfies
  \begin{equation}
      \partial_{xz}u(t,0,0) = 0,
  \end{equation}
  and, as long as the solution remains smooth at time $t$, we have
  \begin{equation}
      \partial_{xzz}u(t,0,z) < 0, \; \text{ for all } z\in[0,1].
  \end{equation}
 In other words, condition (\ref{IC-Wong-4}) is invariant in time.
\end{lemma}
\begin{proof}
  For arbitrary $y_0\in \mathbb{R}$ and $z_0\in[0,1]$, consider the following system of characteristic equations
\begin{equation} \label{ODE}
    \begin{cases}
    \frac{dX}{dt} = u(t,X,Z),\\
    \frac{dY}{dt} = v(t,X,Z), \\
    \frac{dZ}{dt} = w(t,X,Z)
    \end{cases}
\end{equation}
with the initial data
\begin{equation} \label{ODE-IC}
    \begin{cases}
    X(0) = 0,\\
    Y(0) = y_0,\\
    Z(0) = z_0.
    \end{cases}
\end{equation}
By virtue of oddness of $u$ and $v$ in the $x$ variable, the solution $(X,Y,Z)$ must satisfies
\begin{equation}
    X(t) \equiv 0, \;\;\;\; Y(t)\equiv y_0.
\end{equation}
It means that particles starting from the line segment
\begin{equation}
   L: = \Big \{(x,y,z): x=0,\; y=y_0, \;z\in[0,1] \Big \}
\end{equation}
can only move along this line segment. Moreover, when $z_0 = 0$ or $z_0 = 1$, thanks to the boundary condition (\ref{BC-1}), the solution must satisfy additionally $Z(t)\equiv 0$ or $Z(t)\equiv 1$, respectively. This means that the particles stationed at $(0,y_0,0)$ and $(0,y_0,1)$ do not move. Now we follow the same procedures as in the proof of Theorem \ref{theorem-CINT}, and denote by
\begin{equation}
    W(t,z) = w(t,0,z), \;\;\; V(t,z) = -v_x(t,0,z).
\end{equation}
We obtain again the system
\begin{eqnarray}
&&\hskip-.8in
W_{tz} - (W_z)^2 + WW_{zz} + 2\int_0^1 W_z^2(t,z) dz - \Omega V + \Omega \int_0^1 V(t,z)dz= 0, \label{reduced-case2-1} \\
&&\hskip-.8in
V_{t} - W_z V + WV_{z} + \Omega W_z = 0. \label{reduced-case2-2}
\end{eqnarray}
By virtue of (\ref{EQ2-4}), (\ref{BC-1}) and (\ref{IC-Wong-3}), the corresponding initial and boundary conditions are
\begin{equation}\label{reduced-case2-IC}
    W(0,z) = -\int_0^z \partial_x u_0(0,s) ds, \;\; V(0,z) = \Omega,
\end{equation}
\begin{equation}
    W(t,0)=W(t,1) = 0. \label{reduced-case2-BC}
\end{equation}
From (\ref{reduced-case2-IC}), thanks to Proposition \ref{theorem-uniqueness}, we observe that $V \equiv \Omega$ is the unique solution to equation (\ref{reduced-case2-2}). Plugging this back into equation (\ref{reduced-case2-1}), we obtain
\begin{eqnarray}\label{solution-W-2}
&&\hskip-.8in
W_{tz} - (W_z)^2 + WW_{zz} + 2\int_0^1 W_z^2(t,z) dz = 0.
\end{eqnarray}
By taking one derivative with respect to $z$ of (\ref{solution-W-2}), we have
\begin{equation}\label{one-z-derivative}
     W_{tzz} - W_{z} W_{zz} + W W_{zzz} = 0.
\end{equation}
From (\ref{IC-Wong-4}), (\ref{reduced-case2-IC}) and (\ref{reduced-case2-BC}), we know that $W_{zz}(0,0) = 0$ and $W(t,0) = 0$. These together with (\ref{EQ2-4}) and (\ref{one-z-derivative}) imply that
\begin{equation}
    \partial_{xz}u(t,0,0) = W_{zz}(t,0) = 0.
\end{equation}
By taking two derivatives with respect to $z$ of (\ref{solution-W-2}), we have
\begin{equation}\label{two-z-derivative}
     W_{tzzz} - W_{zz}^2 + W W_{zzzz} = 0.
\end{equation}
Since the particles on the line segment $L$ only move along this line, therefore, (\ref{two-z-derivative}) implies
\begin{equation}\label{two-z-derivative-2}
    \frac{d}{dt} W_{zzz}(t,Z(t)) = \frac{d}{dt} w_{zzz}(t,0,Z(t)) = W_{tzzz}(t,Z(t)) + WW_{zzzz}(t,Z(t)) = W_{zz}^2(t,Z(t)) \geq 0.
\end{equation}
Let $\mathcal{T}>0$ such that the solution $(u,v,w,p)$ of system (\ref{EQ2-1})--(\ref{EQ2-4}) remains smooth on $[0,\mathcal{T}]$. Then (\ref{two-z-derivative-2}) implies that as long as $W_{zzz}(0,Z(0)) >0$, we have $W_{zzz}(\mathcal{T},Z(\mathcal{T}))>0$. In order to show that $W_{zzz}(\mathcal{T},z^*) >0$ for each $z^*\in[0,1]$, we need to find $z_0\in[0,1]$ such that $Z(0) = z_0 $ and $Z(\mathcal{T}) = z^*$. For this purpose, we define
\begin{equation}\label{change-variable}
     \tau = \mathcal{T} - t, \;\;\ \tilde{Z}(\tau) = Z(t).
\end{equation}
Then, we have the following ordinary differential equation
\begin{equation}\label{ztilde}
    \frac{d\tilde{Z}(\tau)}{d\tau} = \frac{dZ(t)}{dt} \frac{dt}{d\tau} = -\frac{dZ(t)}{dt} = -W(t,Z(t)) = -W(\mathcal{T}-\tau, \tilde{Z}(\tau)),
\end{equation}
with initial condition
\begin{equation}
    \tilde{Z}(0) = Z(\mathcal{T}) = z^*.
\end{equation}
Since $W$ is smooth on $t\in[0,\mathcal{T}]$, we have a unique solution $\tilde{Z}(\tau)$ on $\tau\in[0,\mathcal{T}]$. Define $z_0 := \tilde{Z}(\mathcal{T})$, then we see from (\ref{change-variable}) that $Z(0) = \tilde{Z}(\mathcal{T}) = z_0 $ and $Z(\mathcal{T}) = \tilde{Z}(0) = z^*$. From (\ref{EQ2-4}) we know that $\partial_{xxz} u(t,0,z) = - W_{zzz}(t,z)$, therefore,
\begin{equation}
      \partial_{xzz}u(t,0,z) < 0, \; \text{ for all } z\in[0,1].
  \end{equation}
\end{proof}
We also need the following lemma. For details, see Lemma 2.5 in \cite{W12}.
\begin{lemma}\label{theorem-lemma-Wong}(see \cite{W12} Lemma 2.5)
Let $f:[0,1] \rightarrow \mathbb{R}$ be a $C^2$ function with the following properties:
\begin{enumerate}
    \item $f'(0) = 0$ and $f''(z)>0$ for any $z\in[0,1]$,
    \item $\int_0^1 f(z) dz = 0$.
\end{enumerate}
Then $f(1) > 0$ and
\begin{equation} \label{inequality}
    \int_0^1 f^2(z) dz \leq \frac{1}{3} f(1)^2.
\end{equation}
\end{lemma}

Now let us return to the proof of Theorem \ref{theorem-Wong}.
\begin{proof} (Proof of Theorem \ref{theorem-Wong})
From the proof in Lemma \ref{theorem-lemma-Wong-invariant}, we know under the assumptions in Theorem \ref{theorem-Wong},
we have
\begin{eqnarray}\label{solution-W-3}
&&\hskip-.8in
W_{tz} - (W_z)^2 + WW_{zz} + 2\int_0^1 W_z^2(t,z) dz = 0.
\end{eqnarray}
From Lemma \ref{theorem-lemma-Wong-invariant}, we know $W_z(t,\cdot) = -\partial_{x} u(t,0,\cdot)$ satisfies both conditions in the Lemma \ref{theorem-lemma-Wong}. Therefore, we have $W_z(t,1) > 0$ and $\int_0^1 W_z(t,z)^2 dz \leq \frac{1}{3} W_z(t,1)^2$. Using this inequality in (\ref{solution-W-3}), and restrict at the point $z=1$, thanks to the boundary condition (\ref{reduced-case2-BC}), we have
\begin{equation}
    W_{zt}(t,1) =  W_z(t,1)^2 - 2 \int_0^1 W_z(t,z)^2 dz \geq \frac{1}{3} W_z(t,1)^2.
\end{equation}
Since $W_z(0,1) >0$, it follows that
\begin{equation}
    W_z(t,1) \geq \frac{3 W_z(0,1)}{3 - W_z(0,1)t}.
\end{equation}
Then we see $W_z(t,1)$ blows up before or at the time $\frac{3}{W_z(0,1)} = \frac{-3}{\partial_x u_0(0,1)}$.
Therefore, the solution $(u,v,w,p)$ must blow up at sometime $\mathcal{T}\in (0, \frac{-3}{\partial_x u_0(0,1)}]$.
\end{proof}
\begin{remark}\label{remark-blowup}
The requirements (\ref{IC-Wong-1})--(\ref{IC-Wong-4}) allow the initial condition to be real analytic. As an example, consider $u_0$ and $v_0$ to be:
\begin{equation}\label{initial-data-different-scale}
    u_0(x,z) = \lambda(-z^2+ \frac{1}{3}) \sin x,\;\;\;\; v_0(x,z) = -\Omega \sin x,
\end{equation}
with $\lambda >0$. For analytic initial data, system (\ref{EQ2-1})--(\ref{EQ2-4}) is local well-posed (cf. \cite{ILT20,KTVZ11}). Therefore, for arbitrary $\Omega\in\mathbb{R}$, we have initial data such that the solution of $3D$ inviscid PEs exists, and also blows up in finite time. For initial data $(u_0,v_0)$, notice that $\int_0^1 u_0(x,z) dz = 0$ and $v_0$ is independent of the $z$ variable. This implies that the baroclinic mode of the initial data is $(u_0, 0)$, and the barotropic mode of the initial data is $(0,v_0)$. We know from above that the guaranteed blowup time is
\begin{equation}
   \frac{-3}{\partial_x u_0(0,1)} = \frac{9}{2\lambda}.
\end{equation}
\end{remark}

\section{ill-posedness}
In this section, we first review the results in \cite{HN16,RE09} about the linear and nonliear ill-posedness of the $2D$ hydrostatic Euler equations in the absence of rotation. Consider $\Omega = 0$ and $v_0 = 0$ in system (\ref{EQ2-1})--(\ref{EQ2-4}), by the uniqueness of smooth solution one concludes that $v\equiv 0$. Therefore, in this case, system (\ref{EQ2-1})--(\ref{EQ2-4}) is reduced to
\begin{eqnarray}
&&\hskip-.8in u_t + uu_x + wu_z +
p_x = 0 , \label{2d-hydrostatic-euler-1}  \\
&&\hskip-.8in
p_z =0 ,   \label{2d-hydrostatic-euler-2}  \\
&&\hskip-.8in
u_x + w_z =0 ,  \label{2d-hydrostatic-euler-3}
\end{eqnarray}
which is the $2D$ hydrostatic Euler equations without rotation. We consider the problem in the horizontal channel $\big\{(x,z): 0\leq z\leq 1, x\in \mathbb{T}\big\}$, with initial condition $u_0$, subject to no-normal flow boundary condition on the top and bottom, and periodic in the $x$ variable:
\begin{equation}\label{BC-2d-hydrostatic-euler}
    \begin{split}
        &w(t,x,0) = w(t,x,1) = 0,\\
        &u, w, p \;\text{are periodic in } x \;\text{with period } 1.
    \end{split}
\end{equation}
Observe that
\begin{equation}
    (U, W)(t) = (U(z), 0), \;\;\; P=0 \label{Shear2}
\end{equation}
is a steady state solution of system (\ref{2d-hydrostatic-euler-1})--(\ref{2d-hydrostatic-euler-3}) with boundary condition (\ref{BC-2d-hydrostatic-euler}). Here $U$ is any shear flow depending only on the $z$ variable. Linearizing (\ref{2d-hydrostatic-euler-1})--(\ref{2d-hydrostatic-euler-3}) around (\ref{Shear2}), we have
\begin{eqnarray}
&&\hskip-.8in \tilde{u}_t + U\, \tilde{u}_x + \tilde{w} U' + \tilde{p}_x = 0 , \label{EQL-1}  \\
&&\hskip-.8in
\tilde{p}_z =0 ,   \label{EQL-2}  \\
&&\hskip-.8in
\tilde{u}_x + \tilde{w}_z =0 .  \label{EQL-3}
\end{eqnarray}
Thanks to (\ref{EQL-3}), we introduce the stream function $\psi$ such that
\begin{eqnarray}
&&\hskip-.8in
(\tilde{u},\tilde{w})=(\psi_z, -\psi_x). \label{stream}
\end{eqnarray}
Differentiating (\ref{EQL-1}) with respect to $z$, and thanks to (\ref{EQL-2}) and (\ref{stream}), we have
\begin{eqnarray}
&&\hskip-.8in \partial_t \psi_{zz} + U(z)\psi_{xzz} -\psi_{x} U''(z)  = 0 . \label{EQL-psi}
\end{eqnarray}
Observe that he no-normal flow boundary condition implies the associated boundary conditions
\begin{eqnarray}
&&\hskip-.8in
\psi(t,x,0) = \psi(t,x,1)=0.\label{LBC}
\end{eqnarray}
The ill-posedness of equation (\ref{EQL-psi}) for certain background shear flow $U(z)$ was established in \cite{RE09}. We summarize this result in the following theorem.
\begin{theorem}\label{theorem-Renardy}(see \cite{RE09})
	Equation (\ref{EQL-psi}) with the boundary conditions (\ref{LBC}) has solutions of the form
	\begin{eqnarray}
	&&\hskip-.8in
	\psi(x,z,t) = \chi(z) \exp(2\pi inx-inct)), \label{L1solution}
	\end{eqnarray}
	where $c$ solves the following equation
	\begin{eqnarray}
	&&\hskip-.8in
	\int_{0}^{1}\Big(U(z)-c\Big)^{-2} dz = 0, \label{eqnc}
	\end{eqnarray}
	and $\chi$ is given by
	\begin{eqnarray}
	&&\hskip-.8in
	\chi(z) = K(U(z)-c) \int_{0}^{z}(U(z)-c)^{-2} dz \label{chi}
	\end{eqnarray}
	for some constant $K$. Moreover, there is a smooth shear flow profile $U(z)$ such that there exist purely imaginary root $c=i\beta$ of (\ref{eqnc}), with $0\neq \beta \in \mathbb{R}$. This implies Hadamard instability, since the growth rate is $n\beta$ with $\beta$ is independent of $n$.
\end{theorem}
\begin{remark}
From Theorem \ref{theorem-Renardy}, we see equation (\ref{EQL-psi}) with the boundary conditions (\ref{LBC}) has solutions of the form $\psi(x,z,t) = e^{2\pi inx} e^{n\beta t}\chi(z)$ with $0 \neq \beta \in \mathbb{R}$. Such Kelvin-Helmholtz type instability, which also appears in the context of vortex sheets, implies the linear ill-posedness of the inviscid PEs in Sobolev spaces and in Gevrey class of order $s>1$. This suggests that the suitable space for the well-posedness of the inviscid PEs is Gevrey class of order $s= 1$, which is exactly the space of analytic functions. This is consistent with results in \cite{ILT20,KTVZ11}.
\end{remark}
In \cite{HN16}, the authors consider the nonlinear perturbation of system (\ref{2d-hydrostatic-euler-1})--(\ref{2d-hydrostatic-euler-3}) around (\ref{Shear2})
\begin{eqnarray}
&&\hskip-.8in \tilde{u}_t + \tilde{u} \tilde{u}_x + U\, \tilde{u}_x + \tilde{w}\tilde{u}_z + \tilde{w} U' + \tilde{p}_x = 0 , \label{EQNL-1}  \\
&&\hskip-.8in
\tilde{p}_z =0 ,   \label{EQNL-2}  \\
&&\hskip-.8in
\tilde{u}_x + \tilde{w}_z =0 .  \label{EQNL-3}
\end{eqnarray}
Defining the vorticity $\tilde{\omega}:= \tilde{u}_z$, and differentiating (\ref{EQNL-1}) with respect to $z$, thanks to (\ref{EQNL-2}) and (\ref{EQNL-3}), one obtains
\begin{equation}\label{equation-vorticity}
    \tilde{\omega}_t + \tilde{u} \tilde{\omega}_x + U\tilde{\omega}_x  + \tilde{w}\tilde{\omega}_z + \tilde{w}U'' = 0,
\end{equation}
in which $\tilde{u}:=\psi_z$ and $\tilde{w}:=-\psi_x$, where $\psi$ is the stream function solves
\begin{equation}\label{streamfun}
    \psi_{zz} = \tilde{\omega}, \;\;\; \psi|_{z=0,1} = 0.
\end{equation}
The following theorem from \cite{HN16} establishes the nonlinear ill-posedness of (\ref{equation-vorticity})--(\ref{streamfun}) in any Sobolev space.
\begin{theorem}\label{theorem-HN}(see \cite{HN16}) There exists a
stationary background shear flow $U(z)$ such that the following holds. For all $s\in\mathbb{N}$, $\alpha \in (0,1]$, and $k\in\mathbb{N}$, there are families of solutions $(\tilde{\omega}_\epsilon)_{\epsilon>0}$ of (\ref{equation-vorticity})--(\ref{streamfun}), and corresponding times $t_\epsilon = \mathcal{O}(\epsilon|\log\epsilon|)$, and $(x_0,z_0)\in \mathbb{T}\times (0,1)$ such that
\begin{equation}
    \lim\limits_{\epsilon\rightarrow 0}\frac{\|\tilde{\omega}_\epsilon\|_{L^2([0,t_\epsilon]\times \Omega_\epsilon)}}{\|\tilde{\omega}_\epsilon|_{t=0}\|^\alpha_{H^s(\mathbb{T}\times(0,1))}} = + \infty,
\end{equation}
where $\Omega_\epsilon = B(x_0,\epsilon^k)\times B(z_0,\epsilon^k)$.
\end{theorem}
\begin{remark}
The shear flow $U(z)$ indicated in Theorem \ref{theorem-HN} is the same as the one indicated in Theorem \ref{theorem-Renardy}, and it can be chosen to be analytic. For instance, $U(z) = \tanh{(\frac{z-1/2}{d})}$ for small $d$ as indicated in \cite{CM91}.
\end{remark}
Now, we extend Theorem \ref{theorem-Renardy} and Theorem \ref{theorem-HN} to the case when $\Omega \neq 0$ for system (\ref{EQ2-1})--(\ref{EQ2-4}). We consider system (\ref{EQ2-1})--(\ref{EQ2-4}) in the horizontal channel $\big\{(x,z): 0\leq z\leq 1, x\in \mathbb{R}\big\}$. Observe that in this situation, we can consider the steady state background flow
\begin{equation}\label{background-geostrophic}
    (U,V,W, P) = (U(z), -\Omega x, 0, -\frac{1}{2}\Omega^2 x^2)
\end{equation}
for system (\ref{EQ2-1})--(\ref{EQ2-4}). Here the $x$-direction component $U$ is the shear flow indicated in Theorem \ref{theorem-Renardy} and Theorem \ref{theorem-HN}, the $y$-direction component $V$ is a Couette shear flow, depending on $\Omega$, in the $x$ variable. Observe that this background flow has infinite energy. However, we will consider next the perturbation around this steady state background flow for system (\ref{EQ2-1})--(\ref{EQ2-4}) and obtain
\begin{eqnarray}
&&\hskip-.8in
\tilde{u}_t + \tilde{u} \tilde{u}_x + U\tilde{u}_x + \tilde{w} \tilde{u}_z + \tilde{w} U' + \tilde{p}_x - \Omega \tilde{v} = 0 , \label{perturbation-1} \\
&&\hskip-.8in
\tilde{v}_t + \tilde{u} \tilde{v}_x + U\tilde{v}_x + \tilde{w} \tilde{v}_z = 0, \label{perturbation-2} \\
&&\hskip-.8in
\tilde{p}_z =0 ,   \label{perturbation-3}  \\
&&\hskip-.8in
\tilde{u}_x + \tilde{w}_z =0, \label{perturbation-4}
\end{eqnarray}
with periodic boundary conditions
\begin{equation}\label{BC-perturbation}
    \begin{split}
        &\tilde{w}(t,x,0) = \tilde{w}(t,x,1) = 0, \\
        &\tilde{u}, \tilde{v}, \tilde{w}, \tilde{p} \;\text{are periodic in } x \;\text{with period } 1.
    \end{split}
\end{equation}
In order to show the ill-posedness of system (\ref{perturbation-1})--(\ref{BC-perturbation}) in Sobolev spaces, we assume by contradiction that it is well-posed. Then by uniqueness we see that if $\tilde{v}_0 = 0$, then $\tilde{v}\equiv 0$. Therefore, system (\ref{perturbation-1})--(\ref{perturbation-4}) reduces to system (\ref{EQNL-1})--(\ref{EQNL-3}). If we only consider linear terms, system (\ref{perturbation-1})--(\ref{perturbation-4}) reduces to system (\ref{EQL-1})--(\ref{EQL-3}). It follows directly from Theorem \ref{theorem-Renardy} and Theorem \ref{theorem-HN} that the perturbed system (\ref{perturbation-1})--(\ref{BC-perturbation}) is both linearly and nonlinearly ill-posed in any Sobolev space, and is linearly ill-posed in Gevrey class of order $s>1$. To be more specific, we have:
\begin{theorem}\label{theorem-illposedness}
    The inviscid PEs (\ref{EQ2-1})--(\ref{EQ2-4}) with boundary condition (\ref{BC-1}) is both linearly and nonlinearly ill-posed in Sobolev spaces, and is linearly ill-posed in Gevrey class of order $s>1$, in the sense that the perturbed system (\ref{perturbation-1})--(\ref{BC-perturbation}) around the certain steady state background flow (\ref{background-geostrophic}) is both linear and nonlinearly ill-posed in Sobolev spaces, and is linearly ill-posed in Gevrey class of order $s>1$.
\end{theorem}
\begin{remark}
Although the perturbed system (\ref{perturbation-1})--(\ref{perturbation-4}) is linearly ill-posed in Sobolev spaces and Gevrey class of order $s>1$,  it can be shown that it is locally well-posed in the space of analytic functions since $U(z)$ is analytic, following \cite{ILT20,KTVZ11}.
\end{remark}

\section{Conclusion}
From the results in section 2 and 3, we see that the $3D$ inviscid PEs, with arbitrary $\Omega\in\mathbb{R}$ and with initial data satisfying (\ref{2d-condition}), can form singularities in finite time, and is both linearly and nonlinearly ill-posed in Sobolev spaces and linearly ill-posed in Gevrey class of order $s>1$. From the Kelvin-Helmholtz type of instability of the inviscid PEs, it is natural to consider the question of well-posedness of the inviscid PEs (with or without rotation) in Gevrey class of order $s=1$, i.e., the space of analytic functions (cf. \cite{ILT20} and \cite{KTVZ11}). Moreover, there is no hope to show the global well-posedness of the $3D$ inviscid PEs, even with fast rotation. The optimal result one can expect is that fast rotation prolongs the life-span of the $3D$ inviscid PEs (cf. \cite{ILT20}).

Both approaches in section 2 restrict the system to the line $x=0$, and use the oddness of the solution to simplify the system. This has the same spirit in \cite{EE97}, where the author established the blowup of solutions to Prandtl equation.

By virtue of Remark \ref{remark-blowup}, we know that the guaranteed blowup time of the solution is $\mathcal{T} = \frac{9}{2\lambda}$. For initial data $(u_0, v_0)$ defined in (\ref{initial-data-different-scale}), $(u_0,0)$ and $(0,v_0)$ correspond to the baroclinic and barotropic mode, respectively. When $|\Omega|\gg 1$, we have:
\begin{itemize}
    \item when $\lambda = |\Omega|$, the baroclinic mode satisfies $(u_0,0) \sim |\Omega|$, and the whole initial data satisfies $(u_0,v_0) \sim |\Omega|$. The guaranteed blowup time in this case satisfies $\mathcal{T}\sim \frac{1}{|\Omega|}$;
    \item when $\lambda = 1$, the baroclinic mode satisfies $(u_0,0) \sim 1$, while the whole initial data satisfies $(u_0,v_0) \sim |\Omega|$. The guaranteed blowup time in this case satisfies $\mathcal{T}\sim 1$;
    \item when $\lambda = \frac{1}{|\Omega|}$, this implies a smallness condition on the baroclinic $(u_0,0) \sim \frac{1}{|\Omega|}$, while the whole initial data satisfies $(u_0,v_0) \sim |\Omega|$. The guaranteed blowup time in this case satisfies $\mathcal{T}\sim |\Omega|$.
\end{itemize}

Based on the observations above, one can expect that the lower bound of the life-span of the $3D$ inviscid PEs in the space of analytic functions can be prolonged with fast rotation, and with some smallness conditions on the size of the baroclinic mode. This result will be reported in \cite{ILT20}.

It remains interesting to know whether for arbitrary $\Omega$ there exists a blowup solution with initial data $(u_0,v_0)$ whose barotropic and baroclinic modes are both of order $1$. Moreover, to estimate the corresponding blowup time $\mathcal{T}$ as $|\Omega|\rightarrow \infty$. Observe that if the blowup time $\mathcal{T} \sim 1$ as $|\Omega|\rightarrow \infty$, this would imply that fast rotation does not prolong the life-span of the solution to the $3D$ inviscid PEs unless, as it has been noted above, a smallness condition on the size of the baroclinic mode is met.

\section*{Acknowledgments}
The work of E.S.T.\ was supported in part by the Einstein Stiftung/Foundation - Berlin, through the Einstein Visiting Fellow Program. The work of S.I was supported by NSERC grant (371637-2019).

\end{document}